\documentclass{my_aims}

\usepackage{amsmath}
\usepackage{paralist}

\usepackage[colorlinks=true]{hyperref}
\hypersetup{urlcolor=blue, citecolor=red}

\newtheorem{theorem}{Theorem}[section]

\theoremstyle{definition}
\newtheorem{definition}[theorem]{Definition}
\newtheorem{remark}{Remark}
\newtheorem*{notation}{Notation}
\newtheorem{example}{Example}

\title[Noether's theorem for problems with time delay]{Noether's Symmetry Theorem
for Variational\\ and Optimal Control Problems with Time Delay}

\author[Gast\~{a}o S. F. Frederico and Delfim F. M. Torres]{}



\email{gastao.frederico@ua.pt}
\email{delfim@ua.pt}

\thanks{The first author is supported by FCT post-doc grant SFRH/BPD/51455/2011}

\begin{document}

\maketitle


\centerline{\scshape Gast\~{a}o S. F. Frederico}
\medskip
{\footnotesize
 \centerline{Department of Science and Technology, University of Cape Verde}
 \centerline{Praia, Santiago, Cape Verde}
 \smallskip
 \centerline{Center for Research and Development in Mathematics and Applications}
 \centerline{Department of Mathematics, University of Aveiro, 3810-193 Aveiro, Portugal}
}

\medskip


\centerline{\scshape Delfim F. M. Torres}
\medskip
{\footnotesize
 \centerline{Center for Research and Development in Mathematics and Applications}
 \centerline{Department of Mathematics, University of Aveiro, 3810-193 Aveiro, Portugal}
}


\bigskip

\centerline{To Professor Helmut Maurer on the occasion of his 65th birthday.}


\begin{abstract}
We extend the DuBois--Reymond necessary optimality condition and
Noether's symmetry theorem to the time delay variational setting.
Both Lagrangian and Hamiltonian versions of Noether's theorem are
proved, covering problems of the calculus of variations
and optimal control with delays.
\end{abstract}


\section{Introduction}

The concept of symmetry plays an important role
in Physics and Mathematics \cite{Torres:Kiev2004}.
Symmetries are described by transformations that
applied to a system result in the same object
after the transformation is carried out \cite{Gouveia:Torres:Rocha}.
They are described mathematically by parameter groups of transformations
\cite{Gouveia:Torres:2009}. Their importance ranges
from fundamental and theoretical aspects to concrete applications,
having profound implications on the dynamical behavior of systems
and on their basic qualitative properties \cite{Torres:2004}.
Another fundamental notion is the concept of a conservation law \cite{Gouveia:Torres:2005}.
Typical applications of conservation laws in the calculus of variations
and optimal control involve reducing the number of degrees of freedom,
thus reducing the problem to a lower dimension and facilitating the integration
of the differential equations given by the necessary optimality conditions
\cite{Rocha:Torres,Torres:2002}.

All differential equations in physics have a variational structure.
In other words, the equations of motion of a physical system
are the Euler--Lagrange equations of a certain variational problem. It turns out that
the conservation laws are the result of invariance of the action with respect to a continuous
group of transformations, given by some symmetry principle. The more general expression of
the interrelation between symmetry/variational structure/conservation is given by Noether's theorem.
Noether's theorem asserts that the conservation laws for a system of differential equations
that correspond to the Euler--Lagrange equations of a certain variational problem come
from the invariance of the variational functional with respect to a parameter continuous
group of transformations \cite{Torres:proper}. In the last few decades,
Noether's principle has been formulated in various contexts: see
\cite{Bartos,Cresson,GF:IJTS:07,GF:JMAA:07,GF2010,NataliaNoether}
and references cited therein. Here, we generalize Noether's theorem
to variational and control problems with time delay.

Variational and control systems with delays
in the state and/or control variables play an important role
in the modeling of phenomena in various applied fields \cite{Basin:book}.
The research literature dedicated to the calculus of variations
and optimal control with delays is vast but, to the best of our knowledge,
it does not include a Noether theorem.
For a gentle introduction to control problems with time delay
we refer the reader to the classical references \cite{Kra,ogu}.
Optimal control problems with delays in the state and control variables,
subject to mixed-state and control-state constraints,
are studied in \cite{Bok,GoKeMa,Kharatishvili}.

This article is organized as follows. In Section~\ref{sec:cNT}
we review one of the proofs of the Noether symmetry theorem.
In Section~\ref{sec:delay} we use the Euler--Lagrange equations
with time delay and respective extremals to prove an extension
of Noether's theorem for problems of the calculus of
variations (Theorem~\ref{theo:tnnd}) and optimal control
(Theorem~\ref{thm:NT:OC}) with time delay.
The results are proved by first extending the classical
DuBois--Reymond necessary optimality condition to the
calculus of variations and optimal control with time delay
(Theorem~\ref{theo:cdrnd} and Theorem~\ref{gndb}, respectively).
Two illustrative examples showing the application of our main results
are given in Section~\ref{exa}.


\section{Review of the Classical Noether's Theorem}
\label{sec:cNT}

There are several different ways to prove the classical Noether's theorem
\cite{Gastao:PhD:thesis}. In this section we review one of these proofs.
In Section~\ref{sec:delay} we then show how this approach can be extended
to problems with time delay. We begin by formulating the fundamental problem
of the calculus of variations: to minimize
\begin{equation}
\label{P}
\mathcal{I}[q(\cdot)] = \int_a^b L\left(t,q(t),\dot{q}(t)\right) dt
\end{equation}
under given boundary conditions $q(a)=q_{a}$ and $q(b)=q_{b}$, and
where $\dot{q} = \frac{dq}{dt}$. The Lagrangian $L :[a,b] \times
\mathbb{R}^{n} \times \mathbb{R}^{n} \rightarrow \mathbb{R}$ is
assumed to be a $C^{1}$-function with respect to all arguments,
and admissible functions $q(\cdot)$ are assumed to be $C^2$-smooth.

\begin{definition}[Invariance of \eqref{P}]
\label{def:inva}
Consider the following $s$-parameter group of infinitesimal transformations:
\begin{equation}
\label{eq:tinf}
\begin{cases}
\bar{t} = t + s\eta(t,q) + o(s),\\
\bar{q}(t) = q(t) + s\xi(t,q) + o(s),
\end{cases}
\end{equation}
where $\eta \in C^1\left(\mathbb{R}^{1+n}, \mathbb{R}\right)$
and $\xi \in C^1\left(\mathbb{R}^{1+n}, \mathbb{R}^n\right)$ are given functions.
The functional \eqref{P} is said to be invariant under \eqref{eq:tinf} if
\begin{equation}
\label{eq:inv1}  \int_{t_{a}}^{t_{b}}
L\left(t,q(t),\dot{q}(t)\right)dt =
\int_{\bar{t}(t_a)}^{\bar{t}(t_b)}
L\left(\bar{t},\bar{q}(\bar{t}),\dot{\bar{q}}(\bar{t})\right)d\bar{t}
\end{equation}
for any subinterval $[{t_{a}},{t_{b}}] \subseteq [a,b]$.
\end{definition}

Throughout the text we denote by $\partial_{i}L$ the partial derivative of $L$
with respect to its $i$th argument, $i = 1,2,3$.

\begin{theorem}[Necessary condition of invariance]
\label{theo:cnsi}
If functional \eqref{P} is invariant under
transformations \eqref{eq:tinf}, then
\begin{equation}
\label{eq:cnsi}
\partial_{1}L\left(t,q,\dot{q}\right)\eta
+\partial_{2}L\left(t,q,\dot{q}\right)\cdot\xi
+\partial_{3}L\left(t,q,\dot{q}\right)\cdot\left(\dot{\xi}
-\dot{q}\dot{\eta}\right)+L\left(t,q,\dot{q}\right)\dot{\eta}=0 \, .
\end{equation}
\end{theorem}

\begin{proof}
Since \eqref{eq:inv1} is satisfied for any subinterval
$[{t_{a}},{t_{b}}]$ of $[a,b]$, one can remove
the integral sign in \eqref{eq:inv1} and write the
equivalent equality
\begin{equation}
\label{eq:inv2} L\left(t,q,\dot{q}\right) =
L\left(t+s\eta+o(s),q+s\xi+o(s), \frac{\dot{q}+
s\dot{\xi}+o(s)}{1+s
\dot{\eta}+o(s)}\right)\frac{d\bar{t}}{dt} \, .
\end{equation}
Equation \eqref{eq:cnsi} is obtained differentiating both sides of
condition \eqref{eq:inv2} with respect to parameter $s$ and then putting $s=0$.
\end{proof}

\begin{definition}[Euler--Lagrange extremal]
A function $q(\cdot) \in C^2$ is said to be an extremal of \eqref{P}
if it satisfies the Euler--Lagrange equation
\begin{equation}
\label{eq:el}
\frac{d}{{dt}}\partial_{3} L\left(t,q(t),\dot{q}(t)\right) =
\partial_{2} L\left(t,q(t),\dot{q}(t)\right), \quad t\in [a,b].
\end{equation}
\end{definition}

\begin{definition}[Constant of motion/conservation law]
A quantity $C(t,q(t),\dot{q}(t))$ is said to be a
\emph{constant of motion} for \eqref{P} if
\begin{equation}
\label{eq:conslaw}
\frac{d}{dt}C(t,q(t),\dot{q}(t))=0,
\end{equation}
$t \in [a,b]$, along all extremals $q(\cdot)$ of \eqref{P}.
The equality \eqref{eq:conslaw} is then a conservation law.
\end{definition}

\begin{theorem}[DuBois--Reymond necessary optimality condition]
\label{theo:cdr} If function $q(\cdot)$ is an extremal of
functional (\ref{P}), then
\begin{equation}
\label{eq:cdr}
\partial_{1}
L\left(t,q(t),\dot{q}(t)\right)=\frac{d}{dt}\left\{L\left(t,q(t),\dot{q}(t)\right)
-\partial_{3} L\left(t,q(t),\dot{q}(t)\right)\cdot\dot{q}(t)\right\}\, .
\end{equation}
\end{theorem}

\begin{proof}
Follows by direct calculations,
using the Euler--Lagrange equation \eqref{eq:el}:
\begin{equation*}
\begin{split}
\frac{d}{dt}&\left\{L\left(t,q,\dot{q}\right)-\partial_{3}
L\left(t,q,\dot{q}\right)\cdot\dot{q}\right\}\\
&=\partial_{1} L\left(t,q,\dot{q}\right)+\dot{q}\cdot\left(\partial_{2}
L\left(t,q,\dot{q}\right)-\frac{d}{dt}\partial_{3}
L\left(t,q,\dot{q}\right)\right)\\
&=\partial_{1} L\left(t,q,\dot{q}\right)\, .
\end{split}
\end{equation*}
\end{proof}

\begin{theorem}[Noether's theorem]
\label{theo:tnoe}
If \eqref{P} is invariant under \eqref{eq:tinf}, then
\begin{equation*}
C(t,q,\dot{q}) = \partial_{3} L\left(t,q,\dot{q}\right)\cdot\xi(t,q)
+ \left( L(t,q,\dot{q}) - \partial_{3} L\left(t,q,\dot{q}\right)
\cdot \dot{q} \right) \eta(t,q)
\end{equation*}
is a constant of motion.
\end{theorem}

\begin{proof}
We use the Euler--Lagrange equation \eqref{eq:el}
and the DuBois--Reymond necessary optimality
condition \eqref{eq:cdr} into the necessary
condition of invariance \eqref{eq:cnsi}:
\begin{equation*}
\begin{split}
0 &= \partial_{1} L\left(t,q,\dot{q}\right)\eta+\partial_{2} L\left(t,q,\dot{q}\right)\cdot\xi
+\partial_{3}
L\left(t,q,\dot{q}\right)\cdot\left(\dot{\xi}-\dot{q}\dot{\eta}\right)
+L\left(t,q,\dot{q}\right)\dot{\eta} \\
&=\partial_{2} L\left(t,q,\dot{q}\right)\cdot\xi+\partial_{3}
L\left(t,q,\dot{q}\right)\cdot\dot{\xi}+\partial_{1}
L\left(t,q,\dot{q}\right)\eta
+\dot{\eta}\left(L\left(t,q,\dot{q}\right)-\partial_{3}
L\left(t,q,\dot{q}\right)\cdot\dot{q}\right)\\
&=\frac{d}{{dt}}\partial_{3}
L\left(t,q,\dot{q}\right)\cdot\xi+\partial_{3}
L\left(t,q,\dot{q}\right)\cdot\dot{\xi}+\frac{d}{dt}\left\{L\left(t,q,\dot{q}\right)-\partial_{3}
L\left(t,q,\dot{q}\right)\cdot\dot{q}\right\}\eta\\
&\qquad +\dot{\eta}\left(L\left(t,q,\dot{q}\right)-\partial_{3}
L\left(t,q,\dot{q}\right)\cdot\dot{q}\right)\\
&=\frac{d}{dt}\left\{\partial_{3} L\left(t,q,\dot{q}\right)\cdot\xi +
\left( L(t,q,\dot{q}) -
\partial_{3} L\left(t,q,\dot{q}\right) \cdot \dot{q} \right)
\eta\right\}.
\end{split}
\end{equation*}
\end{proof}


\section{Main Results: Noether type theorems with time delay}
\label{sec:delay}

In Section~\ref{sub1} we prove two important results for variational problems
with time delays: a DuBois--Reymond necessary optimality condition
(Theorem~\ref{theo:cdrnd}) and a Noether theorem (Theorem~\ref{theo:tnnd}).
The results are then extended in Section~\ref{sub2} to the more general
optimal control setting with delays.


\subsection{Calculus of variations with time delay}
\label{sub1}

We begin by formulating the fundamental problem of the calculus of variations
with time delay: to minimize
\begin{equation}
\label{Pe}
\mathcal{I}^{\tau}[q(\cdot)] = \int_{t_{1}}^{t_{2}}
L\left(t,q(t),\dot{q}(t),q(t-\tau),\dot{q}(t-\tau)\right) dt
\end{equation}
subject to
\begin{equation}
\label{Pe2}
q(t)=\delta(t), \quad t\in[t_{1}-\tau,t_{1}],
\end{equation}
where the Lagrangian $L :[t_1,t_2] \times \mathbb{R}^{n} \times \mathbb{R}^{n}
\times \mathbb{R}^{n}\times \mathbb{R}^{n} \rightarrow \mathbb{R}$
is assumed to be a $C^{1}$-function with respect to all its arguments,
admissible functions $q(\cdot)$ are assumed to be $C^2$-smooth,
$t_{1}< t_{2}$ are fixed in $\mathbb{R}$,
$\tau$ is a given positive real number such that $\tau<t_{2}-t_{1}$,
and $\delta$ is a given piecewise smooth function.

\begin{notation}
For convenience, we introduce the operator $[\cdot]_{\tau}$ defined by
$$
[q]_{\tau}(t)=(t,q(t),\dot{q}(t),q(t-\tau),\dot{q}(t-\tau))\,.
$$
\end{notation}

\begin{theorem}[Euler--Lagrange equations with time delay \cite{CD:Agrawal:1997,DH:1968}]
\label{th:EL1}
If $q(\cdot)$ is a minimizer of problem \eqref{Pe}--\eqref{Pe2},
then $q(\cdot)$ satisfies the following
\emph{Euler--Lagrange equations with time delay}:
\begin{equation}
\label{EL1}
\begin{cases}
\frac{d}{dt}\left\{\partial_{3}L[q]_{\tau}(t)+
\partial_{5}L[q]_{\tau}(t+\tau)\right\}
=\partial_{2}L[q]_{\tau}(t)+\partial_{4}L[q]_{\tau}(t+\tau),
\quad t_{1}\leq t\leq t_{2}-\tau,\\
\frac{d}{dt}\partial_{3}L[q]_{\tau}(t) =\partial_{2}L[q]_{\tau}(t),
\quad t_{2}-\tau\leq t\leq t_{2},
\end{cases}
\end{equation}
where $\partial_{i}L$ is the partial derivative of $L$
with respect to its $i$th argument, $i = 1,\ldots,5$.
\end{theorem}

\begin{definition}[Extremals with time delay]
\label{def:scale:ext}
The solutions $q(\cdot)$ of the  Euler--Lagrange equations
\eqref{EL1} are called \emph{extremals with time delay}.
\end{definition}

\begin{definition}[\textrm{cf.} Definition~\ref{def:inva}]
\label{def:invnd}
The functional \eqref{Pe} is said to be invariant under
the $s$-parameter group of infinitesimal transformations
\eqref{eq:tinf} if
\begin{multline}
\label{eq:invnd}
0 = \left.\frac{d}{ds}\right|_{s=0}
\int_{\bar{t}(I)} L\left(t+s\eta(t,q(t)),q(t)+s\xi(t,q(t)),
\frac{\dot{q}(t)+s\dot{\xi}(t,q(t))}{1+s\dot{\eta}(t,q(t))},\right.\\
\left.q(t-\tau)+s\xi(t-\tau,q(t-\tau)),\frac{\dot{q}(t-\tau)
+s\dot{\xi}(t-\tau,q(t-\tau))}{1+s\dot{\eta}(t-\tau,q(t-\tau))}\right)
\left(1+s\dot{\eta}(t,q(t))\right)dt
\end{multline}
for any  subinterval $I \subseteq [t_1,t_2]$.
\end{definition}

Theorem~\ref{thm:CNSI:SCV} establishes a necessary
condition of invariance for \eqref{Pe}. Conditions
\eqref{eq:cnsind1} and \eqref{eq:cnsind2} are used
in the proof of our Noether-type theorem
(Theorem~\ref{theo:tnnd}).

\begin{theorem}[\textrm{cf.} Theorem~\ref{theo:cnsi}]
\label{thm:CNSI:SCV}
If functional \eqref{Pe} is invariant under the one-parameter
group of transformations \eqref{eq:tinf}, then
\begin{multline}
\label{eq:cnsind1}
\int_{t_1}^{t_2-\tau}\Bigl[\partial_{1}
L[q]_{\tau}(t)\eta(t,q) +\left(\partial_{2}
L[q]_{\tau}(t)+\partial_4 L[q]_{\tau}(t+\tau)\right)\cdot\xi(t,q)\\
+\left(\partial_{3}L[q]_{\tau}(t)
+\partial_5L[q]_{\tau}(t+\tau)\right)\cdot\left(\dot{\xi}(t,q)
-\dot{q}(t)\dot{\eta}(t,q)\right)
+ L[q]_{\tau}(t)\dot{\eta}(t,q)\Bigr]dt = 0
\end{multline}
and
\begin{multline}
\label{eq:cnsind2}
\int_{t_2-\tau}^{t_2}\Bigl[\partial_{1}L[q]_{\tau}(t)\eta(t,q)
+\partial_{2}L[q]_{\tau}(t)\cdot\xi(t,q)\\
+\partial_{3}L[q]_{\tau}(t)\cdot\left(\dot{\xi}(t,q)
-\dot{q}(t)\dot{\eta}(t,q)\right)+L[q]_{\tau}(t)\dot{\eta}(t,q)\Bigr]dt = 0.
\end{multline}
\end{theorem}

\begin{proof}
Without loss of generality, we take $I=[t_1,t_2]$.
Then, \eqref{eq:invnd} is equivalent to
\begin{equation}
\label{eq:cnsind3}
\begin{split}
\int_{t_1}^{t_2} \Bigl[ &\partial_{1}L[q]_{\tau}(t)\eta(t,q)
+\partial_{2}L[q]_{\tau}(t)\cdot\xi(t,q)\\
&+\partial_{3}L[q]_{\tau}(t)\cdot\left(\dot{\xi}(t,q)
-\dot{q}(t)\dot{\eta}(t,q)\right)+L[q]_{\tau}(t)\dot{\eta}(t,q)\Bigr]dt\\
&+\int_{t_1}^{t_2}\Bigl[\partial_{4}
L[q]_{\tau}(t)\cdot\xi(t-\tau,q(t-\tau))\\
&+\partial_{5}L[q]_{\tau}(t)\cdot\left(\dot{\xi}(t-\tau,q(t-\tau))
-\dot{q}(t-\tau)\dot{\eta}(t-\tau,q(t-\tau))\right)\Bigr]dt= 0.
\end{split}
\end{equation}
By performing a linear change of variables $t=\sigma+\tau$ in the last integral
of \eqref{eq:cnsind3}, and keeping in mind that $L[q]_{\tau}(t)\equiv 0$ on
$[t_1-\tau,t_1]$, equation \eqref{eq:cnsind3} becomes
\begin{equation}
\label{eq:cnsind}
\begin{split}
\int_{t_1}^{t_2-\tau}\Bigl[&\partial_{1}
L[q]_{\tau}(t)\eta(t,q) +\left(\partial_{2}
L[q]_{\tau}(t)+\partial_4 L[q]_{\tau}(t+\tau)\right)\cdot\xi(t,q)\\
&+\left(\partial_{3}L[q]_{\tau}(t)+\partial_5
L[q]_{\tau}(t+\tau)\right)\cdot\left(\dot{\xi}(t,q)
-\dot{q}(t)\dot{\eta}(t,q)\right)\\
&+L[q]_{\tau}(t)\dot{\eta}(t,q)\Bigr]dt
+ \int_{t_2-\tau}^{t_2}\Bigl[\partial_{1}
L[q]_{\tau}(t)\eta(t,q) +\partial_{2}
L[q]_{\tau}(t)\cdot\xi(t,q)\\
&+\partial_{3}L[q]_{\tau}(t)\cdot\left(\dot{\xi}(t,q)
-\dot{q}(t)\dot{\eta}(t,q)\right)+L[q]_{\tau}(t)\dot{\eta}(t,q)\Bigr]dt = 0.
\end{split}
\end{equation}
Taking into consideration that \eqref{eq:cnsind} holds for an arbitrary subinterval
$I \subseteq [t_1,t_2]$, equations \eqref{eq:cnsind1} and \eqref{eq:cnsind2} hold.
\end{proof}

\begin{definition}[Constant of motion/conservation law with time delay]
\label{def:leicond}
We say that a quantity
$C(t,t+\tau,q(t),q(t-\tau),q(t+\tau),\dot{q}(t),\dot{q}(t-\tau),\dot{q}(t+\tau))$ is
a \emph{constant of motion with time delay $\tau$} if
\begin{equation}
\label{eq:conslaw:td}
\frac{d}{dt} C(t,t+\tau,q(t),q(t-\tau),q(t+\tau),\dot{q}(t),\dot{q}(t-\tau),\dot{q}(t+\tau))= 0
\end{equation}
along all the extremals $q(\cdot)$ with time delay (\textrm{cf.} Definition~\ref{def:scale:ext}).
The equality \eqref{eq:conslaw:td} is then a conservation law with time delay.
\end{definition}

Theorem~\ref{theo:cdrnd} generalizes the DuBois--Reymond necessary
optimality condition (\textrm{cf.} Theorem~\ref{theo:cdr})
to problems of the calculus of variations with time delay.

\begin{theorem}[DuBois--Reymond necessary condition with time delay]
\label{theo:cdrnd}
If $q(\cdot)$ is an extremal with time delay,
then it satisfies the following conditions:
\begin{equation}
\label{eq:cdrnd}
\frac{d}{dt}\left\{L[q]_{\tau}(t)-\dot{q}(t)\cdot(\partial_{3} L[q]_{\tau}(t)
+\partial_{5} L[q]_{\tau}(t+\tau))\right\} = \partial_{1} L[q]_{\tau}(t)
\end{equation}
for $t_1\leq t\leq t_{2}-\tau$, and
\begin{equation}
\label{eq:cdrnd1}
\frac{d}{dt}\left\{L[q]_{\tau}(t)
-\dot{q}(t)\cdot\partial_{3} L[q]_{\tau}(t)\right\}
=\partial_{1} L[q]_{\tau}(t)
\end{equation}
for $t_2-\tau\leq t\leq t_{2}$.
\end{theorem}

\begin{proof}
We only prove the theorem in the interval $t_{1}\leq t\leq t_{2}-\tau$
(the proof is similar in the interval $t_{2}-\tau\leq t\leq t_{2}$).
We derive equation \eqref{eq:cdrnd} as follows:
\begin{equation}
\label{pr}
\begin{split}
\int_{t_1}^{t_2}\frac{d}{dt}&\left[L[q]_{\tau}(t)-\dot{q}(t)\cdot(\partial_{3} L[q]_{\tau}(t)
+\partial_{5} L[q]_{\tau}(t+\tau))\right]dt\\
&=\int_{t_1}^{t_2}\Bigr[\partial_1 L[q]_{\tau}(t)+
\partial_2 L[q]_{\tau}(t)\cdot \dot{q}(t)-\partial_{5} L[q]_{\tau}(t+\tau)\cdot\ddot{q}(t)\\
&\qquad -\frac{d}{dt}\left\{\partial_{3} L[q]_{\tau}(t)
+\partial_{5} L[q]_{\tau}(t+\tau)\right\}\cdot \dot{q}(t)\Bigr]dt\\
&\qquad +\int_{t_1}^{t_2}\left[\partial_4 L[q]_{\tau}(t)\cdot\dot{q}(t-\tau)
+\partial_5 L[q]_{\tau}(t)\cdot\ddot{q}(t-\tau)\right]dt.
\end{split}
\end{equation}
By performing a linear change of variables $t=\sigma+\tau$ in the last integral
of \eqref{pr}, in the interval where $t_{1}\leq t\leq t_{2}-\tau$,
the equation \eqref{pr} becomes
\begin{equation}
\label{eq3}
\begin{split}
\int_{t_1}^{t_2}\frac{d}{dt}&\left[L[q]_{\tau}(t)-\dot{q}(t)\cdot(\partial_{3} L[q]_{\tau}(t)
+\partial_{5} L[q]_{\tau}(t+\tau))\right]dt\\
&=\int_{t_1}^{t_2-\tau}\Bigl(\partial_1 L[q]_{\tau}(t)
+\dot{q}(t)\cdot\bigr(\partial_{2}L[q]_{\tau}(t)
+\partial_{4}L[q]_{\tau}(t+\tau))\\
&\qquad -\frac{d}{dt}\left[\partial_{3}L[q]_{\tau}(t)
+\partial_{5}L[q]_{\tau}(t+\tau)\right] \cdot \dot{q}(t)\Bigr)dt.
\end{split}
\end{equation}
We finally obtain \eqref{eq:cdrnd} by substituting the Euler--Lagrange
equation with time delay \eqref{EL1} into \eqref{eq3}.
\end{proof}

Theorem~\ref{theo:tnnd} establishes an extension of Noether's
theorem to problems of the calculus of variations with time delay.

\begin{theorem}[The Noether symmetry theorem with time delay in Lagrangian form]
\label{theo:tnnd}
If functional \eqref{Pe} is invariant in the
sense of Definition~\ref{def:invnd}, then the quantity
$C(t,t+\tau,q(t),q(t-\tau),q(t+\tau),\dot{q}(t),\dot{q}(t-\tau),\dot{q}(t+\tau))$,
defined by
\begin{multline}
\label{eq:tnnd}
\left(\partial_{3} L[q]_{\tau}(t)
+\partial_{5} L[q]_{\tau}(t+\tau)\right)\cdot\xi(t,q(t))\\
+\Bigl(L[q]_{\tau}-\dot{q}(t)\cdot(\partial_{3} L[q]_{\tau}(t)
+\partial_{5} L[q]_{\tau}(t+\tau))\Bigr)\eta(t,q(t))
\end{multline}
for $t_1\leq t\leq t_{2}-\tau$ and
\begin{equation}
\label{eq:tnnd1}
\partial_{3} L[q]_{\tau}(t)\cdot\xi(t,q(t))
+\Bigl(L[q]_{\tau}-\dot{q}(t)\cdot\partial_{3} L[q]_{\tau}(t)\Bigr)\eta(t,q(t))
\end{equation}
for $t_2-\tau\leq t\leq t_{2}$, is a constant of motion with time delay
(\textrm{cf.} Definition~\ref{def:leicond}).
\end{theorem}

\begin{proof}
We prove the theorem in the interval $t_1\leq t\leq t_{2}-\tau$.
The proof is similar for the interval $t_2-\tau\leq t\leq t_{2}$.
Noether's constant of motion with time delay \eqref{eq:tnnd} follows by
using the DuBois--Reymond condition with time delay \eqref{eq:cdrnd}
(similarly, \eqref{eq:tnnd1} follows by using \eqref{eq:cdrnd1})
and the Euler--Lagrange equation with time delay \eqref{EL1}
into the necessary condition of invariance \eqref{eq:cnsind1}:
\begin{equation}
\label{eq:cnsind11}
\begin{split}
0&=\int_{t_1}^{t_2-\tau}\Bigl[\partial_{1}
L[q]_{\tau}(t)\eta(t,q)+\left(\partial_{2}
L[q]_{\tau}(t)+\partial_4 L[q]_{\tau}(t+\tau)\right)\cdot\xi(t,q)\\
&\quad +\left(\partial_{3}L[q]_{\tau}(t)+\partial_5
L[q]_{\tau}(t+\tau)\right)\cdot\left(\dot{\xi}(t,q)
-\dot{q}(t)\dot{\eta}(t,q)\right)+L[q]_{\tau}(t)\dot{\eta}(t,q)\Bigr]dt\\
&= \int_{t_1}^{t_2-\tau}\Bigl[\frac{d}{dt}\left(\partial_{3}
L[q]_{\tau}(t)+\partial_5L[q]_{\tau}(t+\tau)\right)\cdot\xi(t,q)\\
&\quad +\left(\partial_{3}
L[q]_{\tau}(t)+\partial_5L[q]_{\tau}(t+\tau)\right)\cdot\dot{\xi}(t,q)\\
&\quad +\frac{d}{dt}\left\{L[q]_{\tau}(t)-\dot{q}(t)\cdot(\partial_{3} L[q]_{\tau}(t)
+\partial_{5} L[q]_{\tau}(t+\tau))\right\}\eta(t,q)\\
&\quad +\left\{L[q]_{\tau}(t)-\dot{q}(t)\cdot(\partial_{3} L[q]_{\tau}(t)
+\partial_{5} L[q]_{\tau}(t+\tau))\right\}\dot{\eta}(t,q)\Bigr]dt\\
&= \int_{t_1}^{t_2-\tau}\frac{d}{dt}\Bigl[\left(\partial_{3} L[q]_{\tau}(t)
+\partial_{5} L[q]_{\tau}(t+\tau)\right)\cdot\xi(t,q(t))\\
&\quad +\Bigl(L[q]_{\tau}-\dot{q}(t)\cdot(\partial_{3} L[q]_{\tau}(t)
+\partial_{5} L[q]_{\tau}(t+\tau))\Bigr)\eta(t,q(t))\Bigr]dt = 0.
\end{split}
\end{equation}
Taking into consideration that equation \eqref{eq:cnsind11} holds for
any subinterval $I\subseteq [t_1,t_2]$, we conclude that
\begin{multline*}
\left(\partial_{3} L[q]_{\tau}(t)
+\partial_{5} L[q]_{\tau}(t+\tau)\right)\cdot\xi(t,q(t))\\
+\Bigl(L[q]_{\tau}-\dot{q}(t)\cdot(\partial_{3} L[q]_{\tau}(t)
+\partial_{5} L[q]_{\tau}(t+\tau))\Bigr)\eta(t,q(t))= \text{constant}.
\end{multline*}
\end{proof}


\subsection{Optimal control with time delay}
\label{sub2}

Theorem~\ref{theo:tnnd} gives a Lagrangian formulation
of Noether's principle to the time delay setting.
Now we give a Hamiltonian formulation of Noether's principle
for more general problems of optimal control with time delay (Theorem~\ref{thm:NT:OC}).
The result is obtained as a corollary of Theorem~\ref{theo:tnnd}.

The optimal control problem with time delay is defined as follows: to minimize
\begin{equation}
\label{Pond}
\mathcal{I}^{\tau}[q(\cdot),u(\cdot)] = \int_{t_1}^{t_2}
L\left(t,q(t),u(t),q(t-\tau),u(t-\tau)\right) dt
\end{equation}
subject to the delayed control system
\begin{equation}
 \label{ci}
\dot{q}(t)=\varphi\left(t,q(t),u(t),q(t-\tau),u(t-\tau)\right)
\end{equation}
and initial condition
\begin{equation}
\label{ic}
q(t)=\delta(t),\quad t\in[t_{1}-\tau,t_{1}],
\end{equation}
where $q(\cdot) \in C^{1}\left([t_1-\tau,t_2], \mathbb{R}^{n}\right)$,
$u(\cdot) \in C^{0}\left([t_1-\tau,t_2], \mathbb{R}^{m}\right)$,
the Lagrangian $L : [t_1,t_2] \times \mathbb{R}^{n}\times\mathbb{R}^{m}
\times \mathbb{R}^{n}\times\mathbb{R}^{m} \rightarrow \mathbb{R}$
and the velocity vector $\varphi : [t_1,t_2] \times \mathbb{R}^{n}\times
\mathbb{R}^{m} \times \mathbb{R}^{n} \times \mathbb{R}^{m} \rightarrow \mathbb{R}^n$
are assumed to be $C^{1}$-functions with respect to all their arguments,
$t_{1} < t_{2}$ are fixed in $\mathbb{R}$, and $\tau$ is a given positive
real number such that $\tau<t_{2}-t_{1}$. As before, we assume
that $\delta$ is a given piecewise smooth function.

\begin{remark}
In the particular case when $\varphi(t,q,u,q_\tau,u_\tau) = u$,
problem \eqref{Pond}--\eqref{ic} is reduced to the problem
of the calculus of variations with time delay \eqref{Pe}--\eqref{Pe2}.
\end{remark}

\begin{notation}
We introduce the operators $[\cdot,\cdot]_{\tau}$ and
$[\cdot,\cdot,\cdot]_{\tau}$ defined by
\begin{equation*}
[q,u]_{\tau}(t)=\left(t,q(t),u(t),q(t-\tau),u(t-\tau)\right),
\end{equation*}
where $q(\cdot) \in C^{1}\left([t_1-\tau,t_2], \mathbb{R}^{n}\right)$
and $u(\cdot) \in C^{0}\left([t_1-\tau,t_2], \mathbb{R}^{m}\right)$; and
\begin{equation*}
[q,u,p]_{\tau}(t)=\left(t,q(t),u(t),q(t-\tau),u(t-\tau),p(t)\right),
\end{equation*}
where $q(\cdot) \in C^{1}\left([t_1-\tau,t_2], \mathbb{R}^{n}\right)$,
$p(\cdot) \in C^{1}\left([t_1,t_2], \mathbb{R}^{n}\right)$,
$u(\cdot) \in C^{0}\left([t_1-\tau,t_2], \mathbb{R}^{m}\right)$.
\end{notation}

\begin{theorem}[\cite{DH:1968}]
\label{theo:pmpnd}
If $(q(\cdot),u(\cdot))$ is a minimizer of
\eqref{Pond}--\eqref{ic}, then there exists a
covector function $p(\cdot)\in C^{1}\left([t_1,t_2],
\mathbb{R}^{n}\right)$ such that the following conditions hold:
\begin{itemize}
\item \emph{the Hamiltonian systems with time delay}
\begin{equation}
\label{eq:Hamnd}
\begin{cases}
\dot{q}(t)=\partial_6 H[q,u,p]_{\tau}(t)\\
\dot{p}(t)=-\partial_2 H[q,u,p]_{\tau}(t)
-\partial_4 H[q,u,p]_{\tau}(t+\tau)
\end{cases}
\end{equation}
for $t_{1}\leq t\leq t_{2}-\tau$, and
\begin{equation}
\label{eq:Hamnd2}
\begin{cases}
\dot{q}(t) = \partial_6 H[q,u,p]_{\tau}(t)  \\
\dot{p}(t) = -\partial_2 H[q,u,p]_{\tau}(t)
\end{cases}
\end{equation}
for $t_{2}-\tau\leq t\leq t_{2}$;

\item \emph{the stationary conditions with time delay}
\begin{equation}
\label{eq:CE}
\partial_3 H[q,u,p]_{\tau}(t)+\partial_5 H[q,u,p]_{\tau}(t+\tau)= 0
\end{equation}
for $t_{1}\leq t\leq t_{2}-\tau$, and
\begin{equation}
\label{eq:CE12}
 \partial_3 H[q,u,p]_{\tau}(t)= 0
\end{equation}
for $t_{2}-\tau\leq t\leq t_{2}$;
\end{itemize}
where the Hamiltonian $H$ is defined by
\begin{equation}
\label{eq:Hnd11}
H[q,u,p]_{\tau}(t) = L[q,u]_{\tau}(t) + p(t)
\cdot \varphi[q,u]_{\tau}(t).
\end{equation}
\end{theorem}

\begin{definition}
\label{scale:Pont:Ext}
A triplet $\left(q(\cdot),u(\cdot),p(\cdot)\right)$
satisfying the conditions of Theorem~\ref{theo:pmpnd}
is called a \emph{Pontryagin extremal with time delay}.
\end{definition}

\begin{remark}
The first equation in the Hamiltonian system
\eqref{eq:Hamnd} and \eqref{eq:Hamnd2} is nothing
but the control system with time delay
$\dot{q}(t)=\varphi[q,u]_{\tau}(t)$ given by \eqref{ci}.
\end{remark}

\begin{remark}
In classical mechanics, $p$ is called the \emph{generalized
momentum}. In the language of optimal control \cite{CD:MR29:3316b},
$p$ is known as the adjoint variable.
\end{remark}

\begin{remark}
\label{rem:cp:CV:EP:EEL}
In the particular case when $\varphi(t,q,u,q_\tau,u_\tau) = u$,
Theorem~\ref{theo:pmpnd} reduces to Theorem~\ref{th:EL1}.
We verify this here in the interval $t_{1}\leq t\leq t_{2}-\tau$,
the procedure being similar for the interval $t_{2}-\tau\leq t\leq t_{2}$.
The stationary condition with time delay \eqref{eq:CE} gives
$p(t) = -\partial_3 L[q]_{\tau}(t)-\partial_5 L[q]_{\tau}(t+\tau)$
and the second equation in the Hamiltonian system with time delay \eqref{eq:Hamnd} gives
$\dot{p}(t) =-\partial_2 L[q]_{\tau}(t)-\partial_4 L[q]_{\tau}(t+\tau)$.
Comparing both equalities, one obtains the first Euler--Lagrange equation
with time delay in \eqref{EL1}. In other words, Pontryagin extremals with time delay
(Definition~\ref{scale:Pont:Ext}) are a generalization of the Euler--Lagrange extremals
with time delay (Definition~\ref{def:scale:ext}).
\end{remark}

In classical optimal control, Theorem~\ref{theo:cdr}
(the DuBois--Reymond necessary optimality condition)
is generalized to the equality $\frac{dH}{dt}=\frac{\partial H}{\partial t}$ \cite{CD:MR29:3316b}.
Next, we extend Theorem~\ref{theo:cdrnd} (the DuBois--Reymond necessary condition with time delay)
to the more general optimal control setting with time delay.

\begin{theorem}
\label{gndb}
If a triplet $(q(\cdot),u(\cdot),p(\cdot))$
with $q(\cdot)\in C^{1}\left([t_1-\tau,t_2], \mathbb{R}^{n}\right)$,
$p(\cdot)\in C^{1}\left([t_1,t_2], \mathbb{R}^{n}\right)$,
and $u(\cdot)\in C^{1}\left([t_1-\tau,t_2], \mathbb{R}^{m}\right)$
is a Pontryagin extremal with time delay,
then it satisfies the following condition:
\begin{equation}
\label{eq:cdrn1}
\frac{d}{dt} H[q,u,p]_{\tau}(t)
= \partial_{1} H[q,u,p]_{\tau}(t),
\quad t\in [t_1,t_2].
\end{equation}
\end{theorem}

\begin{proof}
We prove condition \eqref{eq:cdrn1} by direct calculations:
\begin{equation}
\label{pr11}
\begin{split}
\int_{t_1}^{t_2}&\frac{d}{dt}H[q,u,p]_{\tau}(t)dt\\
&=\int_{t_1}^{t_2}\bigr[\partial_1 H[q,u,p]_{\tau}(t)+\partial_2 H[q,u,p]_{\tau}(t)\cdot \dot{q}(t)
+\partial_3 H[q,u,p]_{\tau}(t)\cdot\dot{u}(t)\\
&\quad +\partial_6 H[q,u,p]_{\tau}(t)\cdot \dot{p}(t)\bigr]dt\\
&\quad +\int_{t_1}^{t_2}\bigr[\partial_4 H[q,u,p]_{\tau}(t)\cdot \dot{q}(t-\tau)
+\partial_5 H[q,u,p]_{\tau}(t)\cdot \dot{u}(t-\tau)\bigr]dt
\end{split}
\end{equation}
and by performing a linear change of variables $t=\nu+\tau$ in the last integral
of \eqref{pr11}, equation \eqref{pr11} becomes
\begin{equation}
\label{pr111}
\begin{split}
\int_{t_1}^{t_2}&\frac{d}{dt}H[q,u,p]_{\tau}(t)dt\\
&=\int_{t_1}^{t_2-\tau}\Bigl[\partial_1 H[q,u,p]_{\tau}(t)
+\left(\partial_2 H[q,u,p]_{\tau}(t)+\partial_4 H[q,u,p]_{\tau}(t+\tau)\right)\cdot \dot{q}(t)\\
&\quad +\left(\partial_3 H[q,u,p]_{\tau}(t)+\partial_5 H[q,u,p]_{\tau}(t+\tau)\right)\cdot \dot{u}(t)
+\partial_6 H[q,u,p]_{\tau}(t)\cdot \dot{p}(t)\Bigr]dt\\
&\quad +\int_{t_2-\tau}^{t_2}\Bigl[\partial_1 H[q,u,p]_{\tau}(t)+
\partial_2 H[q,u,p]_{\tau}(t)\cdot \dot{q}(t)\\
&\quad +\partial_3 H[q,u,p]_{\tau}(t)\cdot \dot{u}(t)
+\partial_6 H[q,u,p]_{\tau}(t)\cdot \dot{p}(t)\Bigr]dt.
\end{split}
\end{equation}
We obtain condition \eqref{eq:cdrn1} by substituting
\eqref{eq:Hamnd} and \eqref{eq:CE} into the first integral,
and substituting \eqref{eq:Hamnd2} and \eqref{eq:CE12}
into the second integral of \eqref{pr111}.
\end{proof}

Using the Lagrange multiplier rule, problem \eqref{Pond}--\eqref{ic}
is equivalent to minimizing
\begin{equation}
\label{eq:pcond}
\mathcal{J}[q(\cdot),u(\cdot),p(\cdot)] = \int_{t_{1}}^{t_{2}}
\left[H(t,q(t),u(t),q(t-\tau),u(t-\tau),p(t))-p(t)\cdot\dot{q}(t)\right]dt
\end{equation}
subject to \eqref{ic}, where $H$ is given by \eqref{eq:Hnd11}.
The notion of invariance for \eqref{Pond}--\eqref{ci}
is defined using the invariance of \eqref{eq:pcond}.

\begin{definition}[\textrm{cf.} Definition~\ref{def:invnd}]
\label{def:invnd-co1}
Consider the following $s$-parameter group of infinitesimal transformations:
\begin{equation}
\label{eq:tinfnd}
\begin{cases}
\bar{t} = t + s\eta(t,q,u) + o(s),\\
\bar{q}(t) = q(t) + s\xi(t,q,u) + o(s),\\
\bar{u}(t) = u(t) + s\varrho(t,q,u) + o(s),\\
\bar{p}(t) = p(t) +s\varsigma(t,q,u) + o(s),
\end{cases}
\end{equation}
where $\eta \in C^1\left(\mathbb{R}^{1+n+m}, \mathbb{R}\right)$,
$\xi, \varsigma \in C^1\left(\mathbb{R}^{1+n+m}, \mathbb{R}^n\right)$,
$\varrho \in C^0\left(\mathbb{R}^{1+n+m}, \mathbb{R}^m\right)$ are given functions.
The functional \eqref{eq:pcond} is said to be invariant under \eqref{eq:tinfnd} if
\begin{multline*}
\left.\frac{d}{ds}\right|_{s=0}
\int_{\bar{t}(I)}
\Biggl[H\Biggl(t+s\eta,q+s\xi,u+s\varrho,q(t-\tau)+s\xi_{\tau}(t),u(t-\tau)+s\varrho_{\tau}(t),\\
-(p+s\varsigma)\cdot\frac{\dot{q}+s\dot{\xi}}{1+s\dot{\eta}}\Biggr)
(1+s\dot{\eta})\Biggr] dt = 0
\end{multline*}
for any  subinterval $I \subseteq [t_1,t_2]$, where
$\varrho_{\tau}(t)=\varrho(t-\tau,q(t-\tau),u(t-\tau))$ and
$\xi_{\tau}(t)=\xi\left(t-\tau,q(t-\tau),u(t-\tau)\right)$.
\end{definition}

\begin{definition}
\label{lco}
A quantity $C(t,q(t),q(t-\tau),u(t),u(t-\tau),p(t))$, constant for $t \in [t_1,t_2]$
along any Pontryagin extremal with delay $(q(\cdot),u(\cdot),p(\cdot))$
of problem \eqref{Pond}--\eqref{ic}, is said to be a
\emph{constant of motion with delay} for \eqref{Pond}--\eqref{ic}.
\end{definition}

Theorem~\ref{thm:NT:OC} gives a Noether-type theorem for
optimal control problems with time delay.

\begin{theorem}[The Noether symmetry theorem with time delay in Hamiltonian form]
\label{thm:NT:OC}
If we have invariance in the sense of Definition~\ref{def:invnd-co1}, then
\begin{multline}
\label{eq:tnnd-co}
C(t,q(t),q(t-\tau),u(t),u(t-\tau),p(t))
= -p(t)\cdot\xi\left(t,q(t),u(t)\right)\\
+ H\left(t,q(t),u(t),q(t-\tau),u(t-\tau),p(t)\right) \eta\left(t,q(t),u(t)\right)
\end{multline}
is a constant of motion with delay for \eqref{Pond}--\eqref{ic}.
\end{theorem}

\begin{proof}
The constant of motion with delay \eqref{eq:tnnd-co} is obtained
by applying Theorem~\ref{theo:tnnd} to problem \eqref{eq:pcond}.
\end{proof}

\begin{remark}
The constant of motion with time delay \eqref{eq:tnnd-co}
has the same expression in the two intervals $t_{1}\leq t\leq t_{2}-\tau$
and $t_{2}-\tau\leq t\leq t_{2}$.
\end{remark}

\begin{remark}
For the problem of the calculus of variations \eqref{Pe}--\eqref{Pe2},
the Hamiltonian \eqref{eq:Hnd11} takes the form $H = L + p \cdot u$,
with $u = \dot{q}$ and $p(t) = -\partial_3 L[q]_\tau(t)-\partial_5 L[q]_\tau(t+\tau)$
(\textrm{cf.} Remark~\ref{rem:cp:CV:EP:EEL}). In this case the
constant of motion with delay \eqref{eq:tnnd-co} reduces to
\eqref{eq:tnnd} in the interval $t_{1}\leq t\leq t_{2}-\tau$ and to \eqref{eq:tnnd1}
in the interval $t_{2}-\tau\leq t\leq t_{2}$.
\end{remark}


\section{Examples}
\label{exa}

Now we illustrate the application
of Theorem~\ref{theo:tnnd} (Example~\ref{ex:1})
and Theorem~\ref{thm:NT:OC} (Example~\ref{ex:2}).

\begin{example}
\label{ex:1}
Let us consider the problem
\begin{equation}
\label{exe1}
\begin{gathered}
\mathcal{J}[q(\cdot)] = \int_{0}^{3}
\left(\dot{q}(t) + \dot{q}(t-1)\right)^{2}dt
\longrightarrow \min\\
q(t)=-t, \quad -1\leq t\leq 0,\\
q(3)=2.
\end{gathered}
\end{equation}
Because \eqref{exe1} is autonomous, one has invariance,
in the sense of Definition~\ref{def:invnd}, with
$\eta \equiv 1$ and $\xi \equiv 0$. We obtain from
Theorem~\ref{theo:tnnd} that
\begin{equation}
\label{eq1:ex1:lc1}
\left(\dot{q}(t)+\dot{q}(t-1)\right)^2
- 2 \dot{q}(t) \left(2 \dot{q}(t) + \dot{q}(t-1)+\dot{q}(t+1)\right) = c_1,
\quad t \in [0,2],
\end{equation}
and
\begin{equation}
\label{eq2:ex1:lc1b}
\left(\dot{q}(t)+\dot{q}(t-1)\right)^2
- 2 \dot{q}(t) \left(\dot{q}(t) + \dot{q}(t-1)\right) = c_2,
\quad t \in [2,3],
\end{equation}
along any extremal $q(\cdot) \in C^1\left([-1,3], \mathbb{R}\right)$ of \eqref{exe1},
where $c_1$ and $c_2$ are constants. In this case the conservation laws
\eqref{eq1:ex1:lc1} and \eqref{eq2:ex1:lc1b} can also be obtained
by a direct application of Theorem~\ref{theo:cdrnd}.
\end{example}

\begin{example}
\label{ex:2}
Let us consider an autonomous optimal control problem with time delay,
\textrm{i.e.}, the situation when $L$ and $\varphi$ in \eqref{Pond} and \eqref{ci}
do not depend explicitly on $t$. In this case one has invariance,
in the sense of Definition~\ref{def:invnd-co1}, for
$\eta \equiv 1$ and $\xi = \varrho = \varsigma \equiv 0$.
It follows from Theorem~\ref{thm:NT:OC} that
\begin{equation}
\label{ceg}
H\left(q(t),u(t),q(t-\tau),u(t-\tau),p(t)\right) = \text{constant}
\end{equation}
along any Pontryagin extremal with delay $(q(\cdot),u(\cdot),p(\cdot))$
of the problem. In this example the same conservation law \eqref{ceg}
is obtained from Theorem~\ref{gndb}.
\end{example}


\section*{Acknowledgments}

This work was supported by FEDER funds through COMPETE
(Operational Programme Factors of Competitiveness)
and by Portuguese funds through the Center for Research and Development
in Mathematics and Applications (University of Aveiro) and the Portuguese Foundation
for Science and Technology (FCT), within project PEst-C/MAT/UI4106/2011
with COMPETE number FCOMP-01-0124-FEDER-022690.
The first author was also supported by FCT through
the program \emph{Ci\^{e}ncia Global}.

The authors are grateful to Ryan Loxton for suggestions regarding
improvement of the text, and to two anonymous referees for valuable
and prompt comments, which significantly contributed to the quality of the paper.



\medskip


\end{document}